\documentclass[oneside, a4paper,11pt,reqno]{amsart}

\usepackage{amssymb,amsfonts,amsmath,curves}
\usepackage[T1]{fontenc}
\usepackage{psfrag}

\usepackage{verbatim}
\usepackage{amsmath}
\usepackage{amsfonts}
\usepackage{color}
\usepackage{amssymb}
\usepackage{epsfig}
\numberwithin{equation}{section}

\newtheorem{theo}{Theorem}
\newtheorem{proposition}[theo]{Proposition}
\newtheorem{corollary}[theo]{Corollary}
\newtheorem{lemma}[theo]{Lemma}

\def \E {\mathbb{E}}
\def \P {\mathbb{P}}
\def \N  {\mathbb{N}}
\def \R  {\mathbb{R}}

\def \Z  {\mathbb{Z}}

\title[Random walk in the low disorder regime]{
Asymptotic expansion of the invariant  measure for ballistic random walk in the low disorder regime}

\author{David Campos and Alejandro F. Ram\'\i rez
}

\address{Escuela de Matem\'atica, Universidad de Costa Rica, josedavid.campos@ucr.ac.cr \and
Facultad de Matem\'aticas, Pontificia Universidad Cat\'olica de Chile, aramirez@mat.puc.cl}
\thanks{Both authors have been
partially supported by Iniciativa Cient\'\i fica Milenio NC120062
and by Fondo Nacional de Desarrollo Cient\'\i fico
y Tecnol\'ogico grant 1141094.}

\date{\today}

\begin{document}

\begin{abstract}
We consider a random walk in random
environment in the low disorder regime
on $\mathbb Z^d$.
That is, the probability that the
random walk jumps from a site $x$
to a nearest neighboring site $x+e$
is given by $p(e)+\epsilon \xi(x,e)$, where
$p(e)$ is deterministic, $\{\{\xi(x,e):|e|_1=1\}:x\in\mathbb Z^d\}$
are i.i.d. and $\epsilon>0$ is a 
parameter which is eventually chosen small
enough.  We establish an asymptotic
expansion  in $\epsilon$ for the invariant measure
of the environmental process whenever
a ballisticity condition is satisfied. As an application of
our expansion, we derive a numerical expression up
to first order in $\epsilon$ for the invariant measure
of random perturbations of the simple symmetric random walk
in dimensions $d=2$.
\end{abstract}

\maketitle

\section{Introduction}

We derive an asymptotic expansion for
the invariant measure of
 the environmental process of random walks 
moving on $\mathbb Z^d$
in the low disorder regime 
within
 the spirit of previous expansions of Sabot \cite{Sa04}
for the velocity. Our result is one of the few 
instances where  explicit quantitative information
about the invariant measure of the environmental
process is given for non-reversible random walks
in random environments in dimensions $d\ge 2$.

For $x\in \R^d$ we denote by $|x|_1$ and $|x|_2$ its
$l^1$ and $l^2$ norms respectively. Let $V:=\{e\in\mathbb Z^d: |e|_1=1\}$
and $\mathcal P:=\{p_e:e\in V\}$ where $p_e\ge 0$ and
$\sum_{e\in V}p_e=1$. We define $\Omega:=\mathcal P^{\mathbb Z^d}$
endowed with its Borel $\sigma$-algebra
and denote any $\omega=\{\omega(x):x\in\mathbb Z^d\}\in \Omega$
where
for each $x\in\mathbb Z^d$ we let 
$\omega(x)=\{\omega(x,e):e\in V\}\in\mathcal P$, an {\it environment}.
We now define the {\it random walk in the environment} $\omega$
starting from $x\in\Z^d$
as the Markov chain $\{X_n:n\ge 0\}$ with
state space $\Z^d$ defined by the  transition probabilities

$$
P(X_{n+1}=x+e|X_n=x)=\omega(x,e),
$$
for  $e\in V$. We denote by $P_{x,\omega}$ its law.
Throughout we will assume that the space of environments $\Omega$ is
endowed with a probability measure $\P$. We will call $P_{x,\omega}$
the {\it quenched law} of the random walk, while $P_x:=\int P_{x,\omega}d\P$
the {\it averaged} or  {\it annealed} law of the random walk.
We will  suppose that $\{\omega(x):x\in\mathbb Z^d\}$ are i.i.d.
under $\P$. The law $\P$ is said to be uniformly elliptic if
there exists a $\kappa>0$ such that for all $x\in\mathbb Z^d$ and
$e\in V$,

$$
\P(\omega(x,e)\ge\kappa)=1.
$$

\medskip

\noindent Define $\mathcal P_0:=\{p\in\mathcal P:\min_{e\in V}p(e)>0\}$.
Consider a transition kernel $p_0=\{p_0(e):e\in V\}\in\mathcal P_0$.
For our main result, we will consider laws $\P$ which are
perturbations of a simple random walk which
jumps according to the transition kernel $p_0$:
for each $\epsilon>0$ we define

\begin{equation}
\label{epsilon-condition}
\Omega_{p_0,\epsilon}:=
\left\{\omega\in\Omega: {\rm for}\ {\rm all}\ x\in\mathbb Z^d, e\in V\
{\rm one}\ {\rm has}\ {\rm that}\ 
\left|\omega(x,e)-p_0(e)\right|\le\epsilon\right\}.
\end{equation}
Let us note that for $\epsilon$ small enough, each probability measure
concentrated on $\Omega_{p_0,\epsilon}$  is uniformly elliptic.
Also,  recall the
definition of the local drift for $x\in\mathbb Z^d$ as
$$
d(x,\omega):=\sum_{e\in V}\omega(x,e)e.
$$
For $\omega\in\Omega$, define the canonical
shifts $\{\theta_x:x\in\mathbb Z^d\}$ as
$\theta_x\omega(y):=\omega(x+y)$.
Finally,  define the {\it environmental process}
$\{\bar\omega_n:n\ge 0\}$ starting from $\bar\omega_0=\omega$
as

$$
\bar\omega_n:=\theta_{X_n}\omega.
$$
\medskip

\noindent To state the main result of this article, let us
define for each $\omega\in\Omega$, $x\in\mathbb Z^d$ and $e\in V$,

\begin{equation}
\label{xi}
\xi(x,e):=\frac{1}{\epsilon}\left(\omega(x,e)-p_0(e)\right),
\end{equation}
so that

$$
\omega(x,e)=p_0(e)+\epsilon\xi(x,e),
$$
and

$$
\bar\xi(x,e):=\xi(x,e)-\mathbb E[\xi(x,e)].
$$
Define also

\begin{equation}
\label{pert}
p_\epsilon(e):=p_0(e)+\epsilon\mathbb E[\xi(0,e)].
\end{equation}
Furthermore, define for $n\ge 0$ and $x,y\in\mathbb Z^d$,
 $p_n(x,y)$ as the probability that a 
random walk with transition kernel $p\in\mathcal P$ jumps from $x$ at time $0$
to site $y$ at time $n$, and the corresponding
potential kernel of the reversed random walk

\begin{equation}
\label{jepsilon}
J_{p^*}(x):=\lim_{n\to\infty}\sum_{k=0}^n\left(p_k(0,-x)-p_k(0,0)\right),
\end{equation}
where here for $p\in\mathcal P$, we define  $p^*$ as the transition
kernel $p^*\in\mathcal P$ defined by $p^*(e):=p(-e)$ for $e\in V$. 
Note that for each $p$ which defines a transient random walk,
the above expression can be written as a difference
of a Green function evaluated at different points. 
On the other hand, in the two-dimensional recurrent case,
(\ref{jepsilon}) is equal to the negative of the potential kernel.

In the main result of this article, we establish
an asymptotic expansion for the invariant measure of
random walks in environments whose law is supported
for a given $p_0\in\mathcal P_0$
in $\Omega_{p_0,\epsilon}$ for $\epsilon$ small enough.
To formulate it, we will
assume the following condition on the local drift.
Given $p_0\in\mathcal P_0$, and $\epsilon>0$ 
we will say that a probability measure $\mathbb P$ defined on
$\Omega$ satisfies the local drift condition {\bf (LD)}
 with bound $\epsilon$ if $\mathbb P(\Omega_{p_0,\epsilon})=1$ and

\begin{equation}
\label{HIP}
\E[d(0,\omega)]\cdot e_1\ge \epsilon.
\end{equation}
Whenever the local drift condition {\bf (LD)} is satisfied,
the random walk satisfies Kalikow's condition \cite{Sa04},
and hence by  Theorem 3.1 of Sznitman
and Zerner \cite{SZ99},
 the environmental process
has a marginal law at fixed time which converges in distribution to an invariant
measure. We will call this invariant measure, the
{\it limiting invariant measure} of the environmental process.
Furthermore, given a measure $\mu$ defined on $\Omega$ and a subset
$B\subset\mathbb Z^d$, we
will call {\it the restriction of $\mu$ to $B$} to
the marginal law of $\mu$ in $\mathcal P^B$.

\medskip

\begin{theo}
\label{theorem1} Let $\eta>0$ and $B$  finite subset of $\mathbb Z^d$.
Then, there is an $\epsilon_0>0$
such that whenever $\epsilon\le \epsilon_0$,
 $p_0\in\mathcal P_0$, and $\mathbb P$ satisfies the
local drift condition {\bf (LD)} [c.f. (\ref{HIP})],
the limiting invariant measure $\mathbb Q$ has
a restriction $\mathbb Q_B$ to  $B$ which 
is absolutely continuous with respect to the restriction
$\mathbb P_B$ to $B$ of $\mathbb P$, with a
Radon-Nikodym derivative admiting $\mathbb P$-a.s. the expansion

\begin{equation}
\label{expansion1}
\frac{d\mathbb Q_B}{d\mathbb P_B}=1+\epsilon\sum_{z\in B}\sum_{e\in V}
\bar\xi(z,e) J_{p^*_\epsilon}(e+z)+ O\left(\epsilon^{2-\eta}\right),
\end{equation}
 where $\left|O\left(\epsilon^{2-\eta}\right)\right|\le c_1\epsilon^{2-\eta}$, for
some constant $c_1=c_1(\eta,\kappa,d,B)$ depending only on $\eta$, $\kappa$, $d$ and $B$.

\end{theo}

\medskip

\noindent Expanding   $J_{p^*_\epsilon}$ it is possible rewrite
in dimensions $d\ge 2$, the expansion (\ref{expansion1}).

\medskip

\begin{corollary}
\label{corollary1}
Let $\eta>0$ and $B$  finite subset of $\mathbb Z^d$.
Then, there is an $\epsilon_0>0$
such that whenever $\epsilon\le \epsilon_0$,
 $p_0\in\mathcal P_0$, and $\mathbb P$ satisfies the
local drift condition {\bf (LD)} [c.f. (\ref{HIP})],
the limiting invariant measure $\mathbb Q$
has a restriction $\mathbb Q_B$ to  $B$ which
is absolutely continuous with respect to the restriction
$\mathbb P_B$ to $B$ of $\mathbb P$, with a
Radon-Nikodym derivative admiting $\mathbb P$-a.s. the expansion

\begin{equation}
\label{expansion2}
\frac{d\mathbb Q_B}{d\mathbb P_B}=1+\epsilon\sum_{z\in B}\sum_{e\in V}
\bar\xi(z,e) J_{p^*_0}(e+z)+ O\left(\epsilon^{2-\eta}\right),
\end{equation}
where
 $\left|O\left(\epsilon^{2-\eta}\right)\right|\le c_2\epsilon^{2-\eta}$, for
some constant $c_2=c_2(\eta,\kappa,d,B)$ depending only on $\eta$, $\kappa$, $d$ and $B$.
\end{corollary}
\medskip

From the point of view of its explicitness, a startling consequence of Corollary \ref{corollary1} is stated in Corollary \ref{corollary2}
of section \ref{formal}, where due to the fact that the potential kernel
 of a simple
symmetric random walk in dimension $d=2$ can be recursively computed, we can obtain a numerical expression
up to first order for the limiting invariant measure.
Furthermore, the Radon-Nykodim derivative  (\ref{expansion1}) plays
an important role in local limit theorems (see for example
Theorem 1.11 of \cite{BCR14} valid for $d\ge 4$).

On the other hand, by the fact that the marginal law of the environmental
process converges to the limiting invariant measure,
 and the fact that

$$
X_n-\sum_{i=0}^{n-1}d(0,\bar\omega_i)\qquad n\ge 0,
$$
is a $P_0$-martingale, we can recover
through Theorem \ref{theorem1}
 Sabot's expansion for the velocity \cite{Sa04},
under the local drift condition {\bf (LD)},

\begin{equation}
\label{vex}
v=\int d(0,\omega)d\mathbb Q=d_0+\epsilon d_1+\epsilon^2 d_2^\epsilon+O(\epsilon^{3-\eta}),
\end{equation}
where $d_0:=\sum_{e\in V}ep_0(e)$, $d_1:=\sum_{e\in V}e \mathbb E[\xi(0,e)]$
and

$$
d_2^\epsilon:=\sum_{e\in V} \sum_{e'\in V} C_{e,e'} J_{p_\epsilon^*}(e),
$$
where $C_{e,e'}:=Cov(\xi(0,e),\xi(0,e'))$.

\medskip

 The absolute continuity of the invariant measure $\mathbb Q$
of Theorem \ref{theorem1} with respect to the law of the environment
restricted to finite sets follows from the proof of Theorem 3.1 of Sznitman and Zerner
in \cite{SZ99}. In dimensions $d\ge 4$, since Kalikow's condition
is satisfied, by a result
of Berger, Cohen and Rosenthal \cite{BCR14} (see also
a previous result of Bolthausen and Sznitman \cite{BS02} valid
at low disorder), we also know that $\mathbb Q$ is absolutely
continuous with respect to $\mathbb P$, and in dimensions $d\ge 2$, 
by \cite{RA03}, we
know that it is absolutely continuous with respect to $\mathbb P$
in every forward half space perpendicular to $e_1$.

\medskip

Random perturbations of 
random walks have been already considered (see \cite{G12}
perturbations leading to the Einstein relation
and \cite{BK91,Sz04,BSZ03, Ba14}
for perturbations of the simple symmetric random walk)
Nevertheless, Theorem \ref{theorem1} 
is one of the first results for the model of random walks
in random environment in the non-reversible context giving explicit quantitative
information about the invariant measure of the environmental process. 

The proof of Theorem \ref{theorem1} is based
on adequate expansions of the Green function
of the random walk within the spirit of \cite{Sa04}.
 Nevertheless, a key
ingredient that we have to incorporate here, is to obtain an expression for the
limiting invariant measure in terms accumulation points of a Ces\`aro type average performed
at a random time with a geometric distribution. This is the content
of Proposition 5 in section \ref{cesaro}. Furthermore, it is necessary
to obtain careful  expansions of Green functions of random
walks perturbed at multiple points.

In the next section of this article, we will derive a more explicit
version of Corollary \ref{corollary1} for the case of perturbations
of the simple symmetric random walk in dimension $d=2$. Then, in section
\ref{formal}, we will give a heuristic explanation of the expansions
(\ref{expansion1}) and (\ref{expansion2}) of Theorem \ref{theorem1} and
Corollary \ref{corollary1}. In section \ref{cesaro}, we will derive
an expression for the limiting invariant measure in terms of 
Ces\`aro averages of the marginal laws of the environmental process
up to a geometric stopping time. In section \ref{green}, we will
show how to perturb at a finite number of sites the Green function
of the random walk. The results of section \ref{green} will
be used to expand a typical term of the expression giving
the limiting measure in Proposition 9 of section \ref{local}.
Using this expansion, Theorem \ref{theorem1} is proved in
section \ref{proof}. Finally, Corollary \ref{corollary1} will
be proved in section \ref{end}.

\medskip

\section{Random perturbations of the simple symmetric
random walk in $d=2$}
\label{consequence}

Here we will derive explicit expression for some
marginal laws of the invariant measure through
Theorem \ref{theorem1} and Corollary \ref{corollary1} for
the the case in which the perturbations
are done on a simple symmetric random walk, so that $p_0(e)=\frac{1}{2d}$ for
all $e\in V$. To simplify notation, we will use the 
notation $J:=J_{p_0}$ 
for this choice of $p_0$. 

Firstly, let us note that when $d=2$,  the explicit values

\begin{equation}
\nonumber
J(z)=
\begin{cases}
0 &\rm{for}\ z=(0,0)\\
-1&\rm{for}\  z=(0,\pm 1), (\pm 1, 0)\\
-\frac{4}{\pi} &\rm{for}\ z=(1,\pm 1), (\pm 1, 1)\\
\frac{8}{\pi}-4 &\rm{for}\ z=(0,\pm 2), (\pm 2, 0)\\
\end{cases}
\end{equation}
can be derived (see McCrea and Whipple \cite{McW40} or Spitzer \cite{Sp64}).
We hence obtain the following corollary from Corollary
\ref{corollary1}. Here we define $z_0:=(0,0)$ and $z_1:=(0,1)$.

\medskip

\begin{corollary}
\label{corollary2} Let $p_0$ be the jump probabilities
of a simple symmetric random walk,  $\eta>0$ and $d=2$.
Then, there is an $\epsilon_0>0$
such that whenever $\epsilon\le \epsilon_0$,
 $p_0\in\mathcal P_0$, and $\mathbb P$ satisfies the
local drift condition {\bf (LD)} [c.f. (\ref{HIP})],
the Radon-Nikodym derivative of the
restriction $\mathbb Q_{z_0,z_1}$ to $\{z_0,z_1\}$  of the limiting
invariant measure $\mathbb Q$ with respect to the restriction
$\mathbb P_{z_0,z_1}$ of $\mathbb P$ to $\{z_0,z_1\}$, admits $\mathbb P$-a.s. the following expansion

\begin{equation}
\label{ec1}
\frac{d\mathbb Q_{z_0,z_1}}{d\mathbb P_{z_0,z_1}}=1-
\frac{4}{\pi}\left(\bar\xi(z_1,e_1)+\bar\xi(z_1,-e_1)\right)\epsilon
+\left(\frac{8}{\pi}-4\right)\bar\xi(z_1,e_2)\epsilon
+O\left(\epsilon^{2-\eta}\right).
\end{equation}
In particular, we have that $\mathbb P$-a.s.

\begin{equation}
\label{ec2}
\frac{d\mathbb Q_{z_i}}{d\mathbb P_{z_i}}=
\begin{cases}
1+ O\left(\epsilon^{2-\eta}\right)\quad &{\rm if}\quad  i=0\\
1-
\frac{4}{\pi}\left(\bar\xi(z_1,e_1)+\bar\xi(z_1,-e_1)\right)\epsilon
+\left(\frac{8}{\pi}-4\right)\bar\xi(z_1,e_2)\epsilon
+O\left(\epsilon^{2-\eta}\right) \quad &{\rm if}\quad  i=1.
\end{cases}
\end{equation}
In both (\ref{ec1}) and (\ref{ec2})
 we have that $\left|O\left(\epsilon^{2-\eta}\right)\right|\le c'_2\epsilon^{2-\eta}$, for
some constant $c'_2=c'_2(\eta)$ depending only on $\eta$.
\end{corollary}

\medskip

\noindent Similar estimates can be obtained for the
marginal law of the limiting invariant measure $\mathbb Q$
restricted to other finite subsets of $\mathbb Z^2$ using
the recursive method presented in \cite{McW40} (see also \cite{Sp64}) to compute
  $J$.

\medskip

\section{Formal derivation of the invariant measure perturbative expansion}
\label{formal}
Here we will show how one can formally derive
the expansion (\ref{expansion1}) of Theorem \ref{theorem1}.
Given any $p\in\mathbb P$ defining a non vanishing drift
$\sum_{e\in V}ep(e)\ne 0$, we define for each $x,y\in\mathbb Z^d$
the Green function $G^p(x,y)$ as the expectation of the number
of visits to site $y$ of the random walk starting from site $x$.

Let us write the transition kernel of the
environmental process as

\begin{equation}
\label{rid}
R=R_0+\epsilon A,
\end{equation}
where for $f:\Omega\to\mathbb R$ we define

$$
R_0f(\omega):=\sum_{e\in V}p_\epsilon(e) f(t_e\omega),
$$
and

$$
Af(\omega):=\sum_{e\in V} \bar\xi(0,e)f(t_e\omega).
$$
The invariant measure $\mathbb Q$ satisfies the equality

\begin{equation}
\label{rf}
\int Rfd\mathbb Q=0,
\end{equation}
for every bounded and continuous function $f:\Omega\to\mathbb R$.
If we assume that the Radon-Nikodym derivative $g:=\frac{d\mathbb Q}{d\mathbb P}$
exists and that it has an analytic expansion

$$
g=\sum_{i=0}^\infty \epsilon^i g_i,
$$
with $g_0=1$, substituting this expansion and $R$ in (\ref{rid}) into
(\ref{rf}), and matching powers, we conclude that for each $i\ge 0$ one has that

\begin{equation}
\label{rec}
g_{i+1}=(R^*)^{-1}A^* g_i,
\end{equation}
where for any linear operator $L$, $L^*$ denotes its adjoint with
respect to the measure $\mathbb P$. Now, note that

$$
(R^*)^{-1}f(\omega)=\sum_{z\in \mathbb Z} G^{p_\epsilon^*}(0,z)f(t_z\omega)
=\sum_{z\in \mathbb Z} G^{p_\epsilon}(z,0)f(t_z\omega).
$$
where $p^*_\epsilon(e):=p_\epsilon(e)$ define the jump probabilities
of the time reversal of the random walk with jump probabilities $p$.
Furthermore,

$$
A^*f(\omega):=\sum_{e\in V}\bar\xi(-e,e)f(t_e\omega).
$$
From the recursion (\ref{rec}) it follows that

$$
g_1(\omega)=\sum_{z\in\mathbb Z^d, e\in V}G^{p_\epsilon^*}(0,z)\bar\xi(z-e,e)=
\sum_{z\in\mathbb Z^d, e\in V}\bar\xi(z,e) G^{p_\epsilon^*}(0,z+e).
$$
Expanding $G^{p^*_\epsilon}$ in $\epsilon$ we formally recover
 the first order term of the expansion (\ref{expansion1}).
\medskip

\section{Invariant measure as a geometric Ces\`aro limit}
\label{cesaro}
In analogy with the fact that limit points of Ces\`aro
averages of the environmental process give rise to invariant
measures, here we will show that when such an average is done
according to a geometric stopping time, its limit points
are still invariant measures.
 
For each $\delta>0$, let us consider
the Green function of the random walk
before an independent stopping time
$\tau_\delta$ with geometric distribution
of parameter $1-\delta$, defined
for $x,y\in\mathbb Z^d$ as

$$
g_\delta^\omega(x,y):=E'_{x,\omega}\left[\sum_{n=0}^{\tau_\delta-1}1_y(X_n)\right],
$$
where the expectation $E'_{x,\omega}$ is taken both over the
random walk and over the
random variable $\tau_\delta$.
Define now the probability measure
$\mu_{\delta}$ on $\Omega$ as the unique probability
measure such that for every continuous and bounded
function $f:\Omega\to\mathbb R$ one has that

\begin{equation}
\label{nud}
\int fd\mu_{\delta}=\frac{\sum_{x\in\mathbb Z^d}\mathbb E\left[g_\delta^\omega
(0,x)f(\theta_x\omega)\right]}{
\sum_{x\in\mathbb Z^d}\mathbb E\left[g_\delta^\omega(0,x)\right]}.
\end{equation}
For the following proposition, we do not require
the environment of the random walk to be elliptic,
nor any other assumption on the environment.

\medskip

\begin{proposition} Consider a random walk in random
environment. Then, each accumulation point of the set
of measures $\{\mu_{\delta}:\delta>0\}$  [c.f. (\ref{nud})] as $\delta$ goes to $1$
is an invariant measure of the environmental
process.
\end{proposition}
\begin{proof}
Note that for each $\delta>0$, as in the proof
of Proposition 2 of Sabot \cite{Sa04},
one can prove that the following identity is satisfied

\begin{equation}
\label{identidad}
\frac{E'_{0}\Bigl[\sum_{k=1}^{\tau_{\delta}-1} f \left(t_{X_{k}}\omega \right) \Bigr]}{E[\tau_{\delta}]} =\frac{\sum_{y \in \Z^d}\E \Bigl[g^{\omega}_{\delta}(0,y)f 
\left(\theta_y \omega\right)\Bigr]}{\sum_{y \in \Z^d}\E \Bigl[g^{\omega}_{\delta}(0,y)\Bigr]},
\end{equation}
where the expectation $E'_0$ denotes taking the annealed expectation
and also the expectation $E$ over $\tau_\delta$.
Let $\mu$ be a limit point of $\{\mu_\delta:\delta>0\}$.
Then,  there exists a sequence $\{\delta_k:k\ge 1\}$ such that 
$\lim_{k\to\infty}\delta_k =1$ and such that $\lim_{k\to\infty}\mu_k:=
\lim_{k\to\infty}\mu_{\delta_k}=\mu$
weakly. 
For $h:\Omega\to\mathbb R$, define
the transition kernel as
\begin{equation}
\label{identidad-2}
Rh(\omega)=\sum_{|e|=1}\omega(0,e)h \left(\theta_{e}\omega\right).
\end{equation}
Using the Markov property of the quenched random walk, one can derive that $$E_{0, \omega} [Rf(\theta_{X_m}\omega)]= R^{m+1}f(\omega), \quad \forall m \in \N.$$ 
Hence, by (\ref{identidad}) and (\ref{identidad-2}) we see that
\begin{eqnarray*}
\int Rf\, d\mu_k &=& \frac{E'_{0}\Bigl[\sum_{m=1}^{\tau_{\delta_k}-1} Rf \left(\theta_{X_{m}}\omega \right) \Bigr]}{E[\tau_{\delta_k}]}=
\frac{\E E \left[\sum_{m=1}^{\tau_{\delta_k}-1}R^{m+1}f(\omega)\right]}{E[\tau_{\delta_k}]} \\
&=&\frac{\E E \left[\sum_{m=1}^{\tau_{\delta_k}-1}R^{m}f(\omega)\right]}{E[\tau_{\delta_k}]}+\frac{1}{E[\tau_{\delta_k}]}\cdot\E E \left[R^{\tau_{\delta_k}}f\right]-\frac{1}{E[\tau_{\delta_k}]}\cdot \E  [Rf]\\
&=&\int f \, d \mu_k+\frac{1}{E[\tau_{\delta_k}]}\cdot\E E \left[R^{\tau_{\delta_k}}f\right]-\frac{1}{E[\tau_{\delta_k}]}\cdot \E  [Rf].
\end{eqnarray*}
Taking the limit when $k\to\infty$
and using the fact that the last two terms tend to zero as $k \to \infty$ by the boundeness of $f$ and the fact that $\lim_{k \to \infty}E\left[\tau_{\delta_k}\right]=\infty$, we conclude that

$$
\int Rfd\mu=\int fd\mu.
$$

%$$
%\int f\, d\nu=\lim_{k \to \infty} \frac{\E_{0}\Bigl[\sum_{l=1}^{\tau_{\delta_k}-1} f \left(t_{X_{l}}\omega \right) \Bigr]}{E[\tau_{\delta_k}]},
%$$

\end{proof}

\medskip

To state the next proposition, we recall some of the so called {\it ballisticity conditions},
which have been important in the study of random walks in random environments with
non vanishing velocity. Given $\gamma\in (0,1)$ and $l\in\mathbb S^d$, we say
that condition $(T)_\gamma$ (see \cite{Sz02}) in direction $l$ is satisfied, if there exists a
neighborhood $V$ of $l$ in $\mathbb S^d$ such that for every $l'\in V$ one has that

$$
\limsup_{L\to\infty}\frac{1}{L^\gamma}\log P_0\left(X_{T_{U_{L,l'}}}\cdot l'<0\right)<0,
$$
where $U_{L,l}$ is a slab defined by

$$
U_{L,l}:=\left\{x\in\mathbb Z^d:-L\le x\cdot l\le L\right\}.
$$
We now say that condition $(T')$ in direction $l$ is satisfied if $(T)_\gamma$ in 
direction $l$ is satisfied for every $\gamma\in (0,1)$. On the other hand,
we say that the polynomial condition $(P)_M$ in direction $l$ is satisfied 
(see \cite{BDR14}) if
for all $L\ge c_0$, where

$$
c_0=2^{3(d-1)}\land\exp\left\{2\left(\ln 90+\sum_{j=1}^\infty \frac{\ln j}{2^j}\right)\right\},
$$
one has that

$$
P_0(X_{T_{B_L}}\cdot l<L)\le\frac{1}{L^M},
$$
where

$$
B_L:=\left\{x\in\mathbb Z^d:-\frac{L}{2}\le x\cdot l\le L, |\pi_l x|_\infty\le 25L^3\right\},
$$
and $\pi_l x$ is the orthogonal projection of $x$ on the subspace perpendicular to $l$.

\medskip

\begin{proposition} 
\label{polinomial}
Consider a random walk in a uniformly
elliptic random
environment satisfying the polynomial condition $(P)_M$ for $M\ge 15d+5$.
 Then,
$\mu:=\lim_{\delta\to 1}\mu_\delta$ exists and is
an invariant measure for the environmental process.
Furthermore, the law of the environmental process
at time $n$, converges in distribution to
$\mu$ as $n\to\infty$.
\end{proposition}
\begin{proof} Let us first note that by Theorem 1 of \cite{BDR14}, the polynomial
condition $(P)_M$ with $M\ge 15d+5$  implies condition $(T')$ of \cite{Sz02}.
On the other hand, Theorem 3.1 of \cite{SZ99}, which is formulated under the assumption
that Kalikow's condition is satisfied, is still valid if Kalikow's condition is replaced
by condition $(T')$. Therefore, since $\lim_{\delta\to 1}\tau_\delta=1$ in probability, by Theorem 3.1 of \cite{SZ99}, 
that there exists an invariant measure $\mu$ of the
environmental process such that in probability

$$
\lim_{\delta\to 1}\frac{1}{\tau_\delta}E_0\left[
\sum_{m=1}^{\tau_\delta-1}f(t_{X_k}\omega )\right]
=\int fd\mu.
$$
Since $\tau_\delta/E[\tau_\delta]$ converges in distribution to an
exponential random variable $S$ of parameter $1$, it follows that in distribution

$$
\lim_{\delta\to 1}\frac{1}{E[\tau_\delta]}E_0\left[
\sum_{m=1}^{\tau_\delta-1}f(t_{X_k}\omega )\right]
=S\int fd\mu.
$$
Hence,

$$
\lim_{\delta\to 1}\frac{1}{E[\tau_\delta]}E'_0\left[
\sum_{m=1}^{\tau_\delta-1}f(t_{X_k}\omega )\right]
=\int fd\mu,
$$
which proves the claim.

\end{proof}

\medskip
\section{Green function expansion}
\label{green}

To prove Theorem \ref{theorem1},  we will extend the method
presented by Sabot in \cite{Sa04}, starting with perturbative
estimates for the Green function of the random walk. To do this, we need first the following lemma, which we will use several times.

\begin{lemma}
\label{unif}
For each $e \in V$, $y, z \in \Z^d $, with $y \neq z$ and $\omega \in \Omega$, we have that
\begin{equation}
\label{unif1}
g_{\delta}^\omega(y,z) \geq \delta \kappa g_\delta^\omega (y,z+e).
\end{equation}
\end{lemma}

\begin{proof}
It is enough to see that (\ref{unif1}) is derived from the next equality 
$$
g^\omega_\delta(y,z)=\delta_{y,z}+\delta \sum_{e \in V} g_{\delta}^\omega (y, z+e)\, \omega(z+e,e)
$$
and the fact that the environment is uniformly elliptic. 
\end{proof}

The main result of this section, is the following lemma
which extends Lemma 1 of \cite{Sa04} for perturbations at
one site of the Green function, to perturbation at
multiple sites.

\medskip

\begin{lemma}
\label{gfe}
Consider an environment $\omega\in\Omega$. For  $B\subset \Z^d$ 
consider an environment $\omega^B$  which is a perturbation of $\omega$ in each  of the points in $B$. That  is,  we have for each $e \in V$
$$
\omega^{B}(x,e):=\left\{\begin{array}{ccc}
\omega(x,e) & {\rm if} & x \not \in B, \\
\omega(x,e)+\Delta_x \omega(e) & {\rm if} & x  \in B,  
\end{array}\right.
$$
for some $\left\{\Delta_x \omega(e):e\in V\right\} \in ]-1,1[^V$
    for each $x\in B$.
Let $0<\delta < 1$. Then, for each $y, y' \in \mathbb Z^d$, one has that 

\begin{equation}
\label{Green expansion1}
\left|g^{\omega^{B}}_{\delta}(y,y')-g^{\omega}_{ \delta}(y,y') \right| \leq c_3\, g^{\omega^{B}}_{ \delta} (y,y')
\end{equation}
and 
\begin{eqnarray}
\label{Green expansion2}
& &\left|g^{\omega^{B}}_{ \delta}(y,y')-g^{\omega}_{ \delta}(y,y') \right. \nonumber \\
& & \left.  -\sum_{x \in B}g^{\omega}_{ \delta}(y,x)\sum_{e \in V} \Delta_x \omega(e)\left[\delta g^{\omega}_{\delta}(x+e,y')-g^{\omega}_{\delta}(x,x) \right] \right| \nonumber \\
& & \leq \frac{\left( 2d \sup_{e \in V, x \in B}|\Delta_x \omega(e)|\right)^2}{\kappa^3} \left[1+\frac{n-1}{(\delta\, \kappa)^{\rho(B)}}\right](1+c_3)n\,\,g^{\omega^B}_{ \delta}(y,y'),
\end{eqnarray}
where 
\begin{equation}
\label{constant1}
c_3=c_3(d, \kappa, B):=\frac{2dn \sup_{e \in V, \, x \in B}  |\Delta_x \omega(e)|}{\kappa^2} \left[\frac{2d \sup_{e \in V, \, x \in B}  |\Delta_x \omega(e)|}{\kappa^2} +1\right]^{n-1},
\end{equation}
and
$\rho(B)$ and $n$, are the diameter and the cardinality of $B$, respectively.

\end{lemma}

\begin{proof}

Let us denote by $B_k$ the set $\{x_1, \ldots, x_k\}$, with $B_n=B$ and $B_0=\emptyset$. We can see, as in the proof of Lemma 1 of Sabot \cite{Sa04}, that for each $y, y'\in \Z^d$, the following inequality is satisfied
$$
\Bigl|g_{\delta}^{\omega^{B_{k+1}}}(y,y')-g_{ \delta}^{\omega^{B_k}}(y,y')\Bigr| \leq \frac{2d \sup_{e \in V}|\Delta_{k+1}\omega(e)|}{\kappa^2}g_{\delta}^{\omega^{B_{k+1}}}(y,y'), \quad \forall k=0, 1, \ldots, n-1,
$$
where $\Delta_{m}=\Delta_{x_m}$ for each $m=1, \ldots, n$. From this, it is easy to deduce that
$$
\left|g^{\omega^{B}}_{\delta}(y,y')-g^{\omega}_{\delta}(y,y') \right| \leq \frac{2d \sup_{e \in V, \, x \in B}  |\Delta_x \omega(e)|}{\kappa^2} \sum_{k=0}^{n-1} g_{\delta}^{\omega^{B_{k+1}}}(y,y').
$$
But repeating the same upper bound a finite number of times, one can deduce that
$$
g_{\delta}^{B_{k}}(y,y') \leq \left[\frac{2d \sup_{e \in V, \, x \in B}  |\Delta_x \omega(e)|}{\kappa^2} +1\right]^{n-k}g_{\delta}^{B}(y,y'), \quad \forall k=1, \ldots, n.
$$
Now, (\ref{Green expansion1}) follows easily.
 Meanwhile, in order to prove (\ref{Green expansion2}), we will use (\ref{Green expansion1}) and the two following inequalities, which are valid in any environment $\omega$,
\begin{equation}
\label{Green1}
|\delta g_{\delta}^\omega (z+e, z)-g_{\delta}^\omega(z,z)| \leq \frac{1}{\kappa}, \quad \forall z \in \Z^d, \,\, e \in V,
\end{equation}
and
\begin{equation}
\label{Green2}
|\delta g_{\delta}^\omega (z+e, y')- g_{\delta}^\omega (z,y')| \leq \frac{1}{\kappa^2}\frac{g_{\delta}^\omega (z,y')}{g_{\delta}^\omega(z,z)}, \quad \forall z \in \Z^d, \,\, e \in V
\end{equation}
(for more details, see Lemma 1 of \cite{Sa04}). Now,
 through standard
Green operator expansions, we can obtain the
following second order expansion  of $g_\delta^{\omega^B}$, 
\begin{eqnarray}
\label{Green3}
& &g^{\omega^{B}}_{ \delta}(y,y')-g^{\omega}_{ \delta}(y,y')-\sum_{x \in B} g^{\omega}_{ \delta}(y,x)\sum_{e \in V} \Delta_{x} \omega(e)\left[\delta g^{\omega}_{\delta}(x+e,y')-g^{\omega}_{\delta}(x,x) \right] \nonumber\\
& & = \sum_{x \in B}\sum_{e \in V} g_{\delta}^{\omega}(y,x)\Delta_x \omega(e)\Bigl[\delta g_{\delta}^{\omega}(x+e,x)-g_{\delta}^{\omega}(x,x)\Bigr] \nonumber  \\ 
& & \times \sum_{z \in B}\sum_{e'\in V}\Delta_z \omega(e') \Bigl[\delta g_{\delta}^{\omega^B}(z+e',y')-g_{\delta}^{\omega^B}(z,y')\Bigr].
\end{eqnarray} 
Hence,  with the help of (\ref{Green1}),(\ref{Green2}) and (\ref{Green3}), we can see that

\begin{eqnarray}
\label{Green4}
& & \Bigl| g^{\omega^{B}}_{ \delta}(y,y')-g^{\omega}_{ \delta}(y,y')-\sum_{x \in B} g^{\omega}_{ \delta}(y,x)\sum_{e \in V} \Delta_{x} \omega(e)\left[\delta g^{\omega}_{\delta}(x+e,y')-g^{\omega}_{\delta}(x,x) \right] \Bigr| \nonumber\\
& & \leq \frac{1}{\kappa^3} \left(2d \sup_{e \in V, \, z \in B}|\Delta_z \omega(e)|\right)^2 \sum_{x,z \in B} \frac{g_{\delta}^{\omega}(y,x) g_{\delta}^{\omega^B}(z,y')}{g_{\delta}^{\omega^B}(z,z)} \nonumber \\
& & = \frac{1}{\kappa^3} \left(2d \sup_{e \in V, \, z \in B}|\Delta_z \omega(e)|\right)^2 \left[\sum_{z \in B} \frac{g_{\delta}^{\omega}(y,z) g_{\delta}^{\omega^B}(z,y')}{g_{\delta}^{\omega^B}(z,z)}+\sum_{\overset{x, z \in B} {x \neq z}}\frac{g_{\delta}^{\omega}(y,x) g_{\delta}^{\omega^B}(z,y')}{g_{\delta}^{\omega^B}(z,z)} \right] \nonumber\\
& & \leq \frac{1}{\kappa^3} \left(2d \sup_{e \in V, \, z \in B}|\Delta_z \omega(e)|\right)^2 \times \left(1+\frac{n-1}{(\delta \kappa)^{\rho(B)}}\right) \left[\sum_{z \in B} \frac{g_{\delta}^{\omega}(y,z) g_{\delta}^{\omega^B}(z,y')}{g_{\delta}^{\omega^B}(z,z)}\right],
\end{eqnarray}
where in the last step, for each $z$ we selected a non-random nearest neighbor self-avoiding path from $z$ to $x$ and we used the inequality $g_\delta^\omega (y,z) \geq \delta \kappa g_{\delta}^\omega (y, z+e)$, for each $e  \in V$, which is an easy consequence of (\ref{unif1}) of Lemma \ref{unif}. Now, with the help of (\ref{Green expansion1}), for each $z \in B$ one can deduce that
\begin{equation}
\label{Green5}
\frac{g_\delta^\omega (y,z)}{g_\delta^{\omega^B}(y,z)}\leq 1+c_3, 
\end{equation}
where $c_3$ is defined in (\ref{constant1}). Thus, we can substitute (\ref{Green5}) into (\ref{Green4}) to conclude that
\begin{eqnarray*}
& &\Bigl| g^{\omega^{B}}_{ \delta}(y,y')-g^{\omega}_{ \delta}(y,y')-\sum_{x \in B} g^{\omega}_{ \delta}(y,x)\sum_{e \in V} \Delta_{x} \omega(e)\left[\delta g^{\omega}_{\delta}(x+e,y')-g^{\omega}_{\delta}(x,x) \right] \Bigr|  \\
& & \leq \frac{1}{\kappa^3} \left(2d \sup_{e \in V, \, z \in B}|\Delta_z \omega(e)|\right)^2 \times \left(1+\frac{n-1}{(\delta \kappa)^{\rho(B)}}\right) (1+c_3)n \,  g_{\delta}^{\omega^B}(y,y').
\end{eqnarray*}

\end{proof}

\section{Local function expansions}
\label{local}
In this section we will derive an asymptotic expansion in
the perturbation parameter $\epsilon$ for certain
expectations of the given local function $f$, involving the
Green function of the random walk. In fact, these expectations
are with respect to the so called Kalikow environment
\cite{K81}.

\medskip

\begin{proposition}
\label{expansion}
Let $A$ be a finite fixed subset of $\Z^d$. Consider a continuous function $f$ defined on $\Omega_{p, \epsilon}$, which depends only on sites located at $A$. 
Let $\eta>0$. Then, there exists an $\epsilon_0>0$, 
and a constant
$c_4=c_4(\eta)$, such that for all $0<\epsilon\le\epsilon_0$,
whenever $\delta$ is close enough to $1$,
there is a  function $h_\delta$ such that 
for each $y \in \Z^d$ 
the following identity is satisfied.

\begin{eqnarray} 
\label{1st estimate}
& \frac{\E \Bigl[g^{\omega}_{\delta}(0,y)f \left(\theta_y \omega\right)\Bigr]}{\E \Bigl[g^{\omega}_{\delta}(0,y)\Bigr]} = \E[f]+\epsilon 
\frac{1}{\E \Bigl[g_{\delta}^{\omega}(0,y)\Bigr]}\sum_{z \in A}\sum_{e \in V} {\rm Cov}\,\Bigl[\xi(z,e), f\Bigr] J_{p^*_\epsilon}(z+e) \E \Bigl[g_{\delta}^{\omega}(0,z+y)\Bigr]
\nonumber \\
&  +\E[h_\delta f]\times O(\epsilon^{2-\eta}),
\end{eqnarray}
where for all $0<\epsilon\le \epsilon_0$ we have that 
\begin{equation}
\label{expansion1a}
\Bigl|\E[|h_\delta |f] \Bigr| \le
\E\Bigl[|f|\Bigr], 
\end{equation}
 $|O(\epsilon^{2-\eta})|\le c_4 \epsilon^{2-\eta}$ and
 $J_{p^*_\epsilon}(x)$ is defined in (\ref{jepsilon}).
\end{proposition}

\medskip

Let us now prove Proposition \ref{expansion}.
 Note that trivially we can get
\begin{equation}
\label{fexp0}
\frac{\E \Bigl[g^{\omega}_{\delta}(0,y)f \left(\theta_y \omega\right)\Bigr]}{\E \Bigl[g^{\omega}_{\delta}(0,y)\Bigr]} =\frac{\E \Bigl[g^{\omega}_{\delta}(0,y)\bar{f} \left(\theta_y \omega\right)\Bigr]}{\E \Bigl[g^{\omega}_{\delta}(0,y)\Bigr]}+\E[f],
\end{equation}
where $\bar{f}=f-\E[f]$.
Next, using the independence between $g^{\omega^{A+y}}_{\delta}$ and $f \circ \theta_y$ and the fact that $\bar{f}$ is a centered random variable, we can see that 
\begin{equation}
\label{fexp1}
 \frac{\E \Bigl[g^{\omega}_{\delta}(0,y)\bar{f} \left(\theta_y \omega\right)\Bigr]}{\E \Bigl[g^{\omega}_{\delta}(0,y)\Bigr]}\nonumber  =
 \frac{\E \Bigl[\left(g^{\omega}_{\delta}(0,y)-g^{\omega^{A+y}}_{\delta}(0,y)\right)\bar{f} \left(\theta_y \omega\right)\Bigr]}{\E \Bigl[g^{\omega}_{ \delta}(0,y)\Bigr]}.
\end{equation}
Note that $g_{\delta}^{\omega^{A+y}}$ does not depend on the coordinates
of $\omega$  located at $A+y$. Thus, using inequality (\ref{Green expansion2}) of Lemma \ref{gfe}, we can deduce that
\begin{eqnarray}
\label{fexp2}
& &\frac{\E \Bigl[\left(g^{\omega}_{\delta}(0,y)-g^{\omega^{A+y}}_{\delta}(0,y)\right)\bar{f} \left(\theta_y \omega\right)\Bigr]}{\E \Bigl[g^{\omega}_{ \delta}(0,y)\Bigr]} \nonumber \\
& & =\frac{\epsilon\,\E \left[\sum_{z \in A+y}\sum_{e \in V}g^{\omega^{A+y}}_{\delta}(0,z) \bar{\xi}(z,e) \left(\delta g^{\omega^{A+y}}_{\delta}(z+e,y)-g^{\omega^{A+y}}_{\delta}(z,z)\right)\bar{f}\left(\theta_y \omega\right)\right]}{\E \Bigl[g^{\omega}_{\delta}(0,y)\Bigr]}\nonumber \\ 
& & +
\frac{\E[g_\delta^\omega (0,y)\bar{f}(\theta_y \omega)\times O_1(\epsilon^2)]}{\E[g_\delta^\omega (0,y)]}.
\end{eqnarray}
where $O_1(\epsilon)$ satisfies the inequality
\begin{equation}
\label{fexp2.1}
|O_1(\epsilon^2)| \leq \frac{8d^2}{\kappa^3} \left[1+\frac{n-1}{(\delta\, \kappa)^{\rho(A)}}\right](1+c_5)n \epsilon^2,
\end{equation}
$n$ is the cardinality of $A$ and here $c_5$ is defined by (see (\ref{epsilon-condition}), (\ref{xi}) and (\ref{constant1}))

$$c_5=c_5(d, \kappa, A):=\frac{2dn \epsilon} {\kappa^2} \left[\frac{2d \epsilon}{\kappa^2} +1\right]^{n-1}.$$

\noindent By the independence between $\bar\xi(z,e)$ for $z\in A+y$
and the Green function $g_\delta^{\omega^{A+y}}$, we can see by (\ref{fexp2})
that

\begin{eqnarray}
\label{fexp2.2} 
& &\frac{\E \Bigl[\left(g^{\omega}_{\delta}(0,y)-g^{\omega^{A+y}}_{\delta}(0,y)\right)\bar{f} \left(\theta_y \omega\right)\Bigr]}{\E \Bigl[g^{\omega}_{ \delta}(0,y)\Bigr]} \nonumber \\
& & = \epsilon \sum_{z \in A}\sum_{e \in V} {\rm Cov}\,\Bigl[\xi(z,e), f(\omega)\Bigr] \times \frac{\E \Bigl[g^{\omega^{A+y}}_\delta (0,z+y)\left(\delta g^{\omega^{A+y}}_{\delta}(z+y+e,y)-g^{\omega^{A+y}}_{\delta}(z+y,z+y)\right)\Bigr]}{\E \Bigl[g^{\omega}_{\delta}(0,y)\Bigr]} \nonumber \\ 
& & +\frac{\E[g_\delta^\omega (0,y)\bar{f}(\theta_y \omega)\times O_1(\epsilon^2)]}{\E[g_\delta^\omega (0,y)]}.
\end{eqnarray}
In addition, with the help of (\ref{Green expansion1}) of
Lemma \ref{gfe}, we can thanks to the development of (\ref{fexp2.2})
conclude that

\begin{eqnarray}
\label{fexp3}
 &\frac{\E \Bigl[\left(g^{\omega}_{\delta}(0,y)-g^{\omega^{A+y}}_{\delta}(0,y)\right)\bar{f} \left(\theta_y \omega\right)\Bigr]}{\E \Bigl[g^{\omega}_{ \delta}(0,y)\Bigr]} \nonumber \\
 & = \epsilon \sum_{z \in A}\sum_{e \in V} {\rm Cov}\,\Bigl[\xi(z,e), f(\omega)\Bigr] \times \frac{\E \Bigl[g^{\omega^{A+y}}_\delta (0,z+y)\left(\delta g^{\omega^{A+y}}_{\delta}(z+y+e,y)-g^{\omega^{A+y}}_{\delta}(z+y,z+y)\right)\Bigr]}{\E \Bigl[g^{\omega^{A+y}}_{\delta}(0,y)\Bigr]} \nonumber \\
 & + \sum_{z \in A}\sum_{e \in V} {\rm Cov} \, \Bigl[\xi(z,e), f(\omega)\Bigr]\times\frac{\E \Bigl[g^{\omega^{A+y}}_\delta (0,z+y)\left(\delta g^{\omega^{A+y}}_{\delta}(z+y+e,y)-g^{\omega^{A+y}}_{\delta}(z+y,z+y)\right)\Bigr]}{\E \Bigl[g^{\omega^{A+y}}_{\delta}(0,y)\Bigr]} \nonumber \\
 & \times \frac{\E[g_\delta^\omega (0,y)O_2(\epsilon^2)]}{\E[g^\omega_\delta (0,y)]}+ \frac{\E[g_\delta^\omega (0,y)\bar{f}(\theta_y \omega) O_1(\epsilon^2)]}{\E[g_\delta^\omega (0,y)]},
\end{eqnarray}
where 
\begin{equation}
\label{fexp3.1}
|O_2(\epsilon^2)| \leq \left(\frac{4dn}{\kappa^2}\Bigl[\frac{2d \epsilon}{\kappa^2}+1\Bigr]^{n-1}\right)\epsilon^2.
\end{equation}
Now, to express the second term of (\ref{fexp3}) in
terms of  $J_{p^*_\epsilon}$ [c.f. (\ref{jepsilon})], we will 
require a lemma which is a variation of Lemma 3 of \cite{Sa04}.
For $v\in\mathbb Z^d$, define
\begin{equation}
\label{auxiliar}
\phi^{\epsilon}(v):=\prod_{i=1}^d \left(\sqrt{\frac{p_{\epsilon}(-e_i)}{p_{\epsilon}(e_i)}}\right)^{v_i},
\end{equation}
where $v_i$ are the coordinates of $v$. Also,
for each $z \in A$, $e \in V$ and $y \in \Z^d$, define
\begin{equation}
\label{JGreen}
J_e^{\delta}(y, z):=\frac{\E \Bigl[g^{\omega^{A+y}}_\delta(0,z+y)\Bigl(\delta g^{\omega^{A+y}}_\delta(z+y+e,y)- g^{\omega^{A+y}}_\delta(z+y,z+y)\Bigr)\Bigr]}{\E \Bigl[g^{\omega^{A+y}}_\delta (0,z+y)\Bigr]}.
\end{equation}

\medskip

\medskip
\begin{lemma}
\label{Sabot}
Assume that the measure $\mathbb P$ satisfies the
local drift condition {\bf (LD)} [c.f. (\ref{HIP})]. Let $\eta>0$. Then there exists a constant $c_6=c_6(\eta)>0$ and $\epsilon_0>0$
such that for each $\epsilon\le\epsilon_0$ we have that for all $z \in A$, $e \in V$
and $y \in \Z^d$ one has that

\begin{equation}
\label{Sabot1}
\varlimsup_{\delta \to 1} |J^\delta_e (y,z)-J_{p^*_\epsilon}(z+e)| \leq c_6\phi^\epsilon(z+e)\epsilon^{1-\eta}.
\end{equation}
\end{lemma}

\begin{proof} We will just give an outline of the proof, stressing the
steps where modifications have to be made with respect to the
proof of Lemma 3 of \cite{Sa04}.
For each $z \in A$, $y \in \Z^d$ we define
$$
\tilde{\P}:=\frac{g_\delta^{\omega^A}(0, y+z)}{\E \Bigl[g_\delta^{\omega^A}(0, y+z)\Bigr]} \P.
$$
Now, using a generalized version of a result of Kalikow \cite{K81}, stated in \cite{Sa04}, we can see that
$$
J_e^\delta(y,z)=\delta g_\delta^{\tilde{\omega}}(z+y+e, y)-g_\delta^{\tilde{\omega}}(z+y, z+y),
$$
where $g_\delta^{\tilde{\omega}}$ denotes the Green function of 
 Kalikow random walk, defined by its transition probabilities 
$\tilde{\omega}(x,e)$ given by
$$
\tilde{\omega}(x,e):=\frac{\tilde{\E}\Bigl[g_\delta^{\omega^A}(y+z,x)\omega^A (x,e)\Bigr]}{\tilde{\E}\Bigl[g_\delta^{\omega^A}(y+z,x)\Bigr]},
$$
for each $x \in \Z^d$ and $e \in V$. Here $\tilde{\E}$ is the expectation with respect to $\tilde{\P}$. It is easy to verify that
\begin{equation}
\label{env1}
\tilde{\omega}(x,e)=\left\{\begin{array}{ccc}
\E[\omega(x,e)]+\epsilon \frac{\tilde{\E}\Bigl[g_\delta^{\omega^A}(y+z,x)\bar{\xi} (x,e)\Bigr]}{\tilde{\E}\Bigl[g_\delta^{\omega^A}(y+z,x)\Bigr]} & {\rm if} & x \not \in A, \\
\E[\omega(x,e)] & {\rm if} & x \in A.
\end{array}
\right.
\end{equation}
Using twice (\ref{Green5}), we can deduce that
\begin{eqnarray}
\label{env2}
& \frac{\tilde{\E}\Bigl[g_\delta^{\omega^A}(y+z,x)\bar{\xi} (x,e)\Bigr]}{\tilde{\E}\Bigl[g_\delta^{\omega^{A}}(y+z,x)\Bigr]}= \frac{\tilde{\E}\Bigl[g_\delta^{\omega^{A \cup \{x\}}}(y+z,x)\bar{\xi} (x,e)\Bigr]}{\tilde{\E}\Bigl[g_\delta^{\omega^{A \cup \{x\}}}(y+z,x)\Bigr]}+O(\epsilon) \nonumber \\
& = \frac{\E \Bigl[g_\delta^{\omega^A}(0,y+z)\, g_\delta^{\omega^{A \cup \{x\}}}(y+z,x)\,\bar{\xi} (x,e)\Bigr]}{\E\Bigl[g_\delta^{\omega^{A}}(0,y+z)\,g_\delta^{\omega^{A \cup \{x\}}}(y+z,x)\Bigr]}+O(\epsilon)\nonumber \\
&=O(\epsilon),
\end{eqnarray}
where in the last step, we used the independence of $g_\delta^{A\cup \{x\}}$ with $\bar{\xi}(x,e)$ and $|O(\epsilon)| \leq c_7 \epsilon$, where $c_7$ is a constant, which depends on $\kappa, d$ and cardinality of $A$ (see (\ref{constant1})). Now, from (\ref{env1}) and (\ref{env2}), we can deduce that for each $x \in \Z^d$ and $e \in V$ the following identity is satisfied
$$
\tilde{\omega}(x,e)=\E[\omega(x,e)]+\epsilon^2 \Delta \omega(x,e),
$$ 
where $\Delta \omega(x,e)$ is uniformly bounded in $x, e, y, z, \delta, \epsilon$. 
The following steps of the proof are identical to steps 2 and 3 of Lemma 3 of \cite{Sa04}.
\end{proof}

\medskip

We can now continue with the proof of Proposition \ref{expansion}.
Note that using the definition (\ref{JGreen}), we can
rewrite  (\ref{fexp3}) as

\begin{eqnarray}
\label{fexp3.2}
 &\frac{\E \Bigl[\left(g^{\omega}_{\delta}(0,y)-g^{\omega^{A+y}}_{\delta}(0,y)\right)\bar{f} \left(\theta_y \omega\right)\Bigr]}{\E \Bigl[g^{\omega}_{ \delta}(0,y)\Bigr]} \nonumber \\
 & = \epsilon \sum_{z \in A}\sum_{e \in V} {\rm Cov}\,\Bigl[\xi(z,e), f(\omega)\Bigr] J_e^\delta (y,z) \times \frac{\E \Bigl[g_\delta^{\omega^{A+y}}(0, z+y)\Bigr]}{\E \Bigl[g^{\omega^{A+y}}_{\delta}(0,y)\Bigr]} \nonumber \\
 &+ \sum_{z \in A}\sum_{e \in V} {\rm Cov} \, \Bigl[\xi(z,e), f(\omega)\Bigr] J_e^\delta (y,z)\times \frac{\E[g_\delta^{\omega^{A+y}} (0,z+y)]}{\E[g_\delta^{\omega^{A+y}} (0,y)]}\times \frac{\E[g_\delta^\omega (0,y)O_2(\epsilon^2)]}{\E[g^\omega_\delta (0,y)]} \nonumber \\
 & + \frac{\E[g_\delta^\omega (0,y)\bar{f}(\theta_y \omega) O_1(\epsilon^2)]}{\E[g_\delta^\omega (0,y)]}.
\end{eqnarray}
On the other hand, with the help of (\ref{Sabot1}) of Lemma \ref{Sabot}, it follows that for each $\eta>0$ we can choose  $\delta_0$ 
such that for $\delta\ge\delta_0$ one has that

\begin{equation}
\label{lemma11}
 |J^\delta_e (y,z)-J_{p^*_\epsilon}(z+e)| \leq 2c_{6}\phi^\epsilon(z+e)\epsilon^{1-\eta}.
\end{equation}
Hence for  $\delta\ge\delta_0$ we conclude from (\ref{fexp3.1})  that

\begin{eqnarray}
\label{fexp4}
 &\frac{\E \Bigl[\left(g^{\omega}_{\delta}(0,y)-g^{\omega^{A+y}}_{\delta}(0,y)\right)\bar{f} \left(\theta_y \omega\right)\Bigr]}{\E \Bigl[g^{\omega}_{ \delta}(0,y)\Bigr]} \nonumber \\
 & = \epsilon \sum_{z \in A}\sum_{e \in V} {\rm Cov}\,\Bigl[\xi(z,e), f(\omega)\Bigr]J_{p_\epsilon^*}(z+e) \times \frac{\E \Bigl[g_\delta^{\omega^{A+y}}(0, z+y)\Bigr]}{\E \Bigl[g^{\omega^{A+y}}_{\delta}(0,y)\Bigr]} \nonumber \\
& + O_3(\epsilon^{2-\eta}) \sum_{z \in A}\sum_{e \in V} {\rm Cov}\,\Bigl[\xi(z,e), f(\omega)\Bigr] \nonumber \\
 &+ \sum_{z \in A}\sum_{e \in V} {\rm Cov} \, \Bigl[\xi(z,e), f(\omega)\Bigr] J_{p_\epsilon^*}(z+e)\,\, \frac{\E[g_\delta^\omega (0,y)O_4(\epsilon^2)]}{\E[g^\omega_\delta (0,y)]} \nonumber\\
 &+ \sum_{z \in A}\sum_{e \in V} {\rm Cov} \, \Bigl[\xi(z,e), f(\omega)\Bigr] \,\, \frac{\E[g_\delta^\omega (0,y)O_5(\epsilon^{3-\eta})]}{\E[g^\omega_\delta (0,y)]} \nonumber\\ 
 &+ \frac{\E[g_\delta^\omega (0,y)\bar{f}(\theta_y \omega) O_1(\epsilon^2)]}{\E[g_\delta^\omega (0,y)]}.
\end{eqnarray}
where we have used the fact that the expression
 $\frac{\E \Bigl[g_\delta^{\omega^{A+y}}(0, z+y)\Bigr]}{\E \Bigl[g^{\omega^{A+y}}_{\delta}(0,y)\Bigr]}$ can be bounded by $\frac{1}{(\delta\, \kappa)^{\rho(A)}}$  choosing a non-random nearest neighbor self-avoiding path from $y$ to $z+y$ and using (\ref{unif1}) of Lemma \ref{unif} at most $\rho(A)$ times, and where

\begin{equation}
\label{fexp4.1}
|O_3(\epsilon^{2-\eta})| \leq \frac{2c_{6} c_8 \epsilon^{2-\eta}}{(\delta \kappa)^{\rho(A)}},  \quad |O_4(\epsilon^2)| \leq \frac{\left(\frac{4dn}{\kappa^2}\Bigl[\frac{2d \epsilon}{\kappa^2}+1\Bigr]^{n-1}\right)}{(\delta \kappa)^{\rho(A)}}\epsilon^2,
\end{equation}
\begin{equation}
\label{fexp4.2}
|O_5(\epsilon^{3-\eta})| \leq 2 \frac{\left(\frac{4dn}{\kappa^2}\Bigl[\frac{2d \epsilon}{\kappa^2}+1\Bigr]^{n-1}\right)}{(\delta \kappa)^{\rho(A)}}c_{6} c_8\epsilon^{3-\eta}.
\end{equation}
and
$$
c_:=\sup_{z \in A, \, e \in V}|\phi^{\epsilon}(z+e)|.
$$
In addition, if we use again (\ref{Green expansion1}) of Lemma \ref{gfe}, we can say that
\begin{equation}
\label{fexp5}
\frac{\E \Bigl[g_\delta^{\omega^{A+y}}(0, z+y)\Bigr]}{\E \Bigl[g^{\omega}_{\delta}(0,z+y)\Bigr]}=1+\frac{\E \Bigl[g^\omega_\delta (0,z+y)O_6(\epsilon)\Bigr]}{\E \Bigl[g^\omega_\delta (0,z+y)\Bigr]},
\end{equation}
and
\begin{equation}
\label{fexp5a}
\frac{\E \Bigl[g_\delta^{\omega}(0,y)\Bigr]}{\E \Bigl[g^{\omega^{A+y}}_{\delta}(0,y)\Bigr]}=\frac{\E \Bigl[g^\omega_\delta (0,y)\Bigr]}{\E \Bigl[g^\omega_\delta (0,y)\Bigl(1+O_7(\epsilon)\Bigr)\Bigr]},
\end{equation}
provided that for each $\tau>0$ one has that
\begin{equation}
\label{fexp5.1}
|O_i(\tau)| \leq c_3 \tau, \quad \forall i=6,7.
\end{equation}
Using (\ref{fexp5}) and (\ref{fexp5a}) for the first term of the right-hand side of (\ref{fexp4}) and using once more  (\ref{lemma11}),
we see that for $\delta\ge \delta_0$ one has that
\begin{eqnarray}
\label{fexp6}
& \frac{\E \Bigl[\left(g^{\omega}_{\delta}(0,y)-g^{\omega^{A+y}}_{\delta}(0,y)\right)\bar{f} \left(\theta_y \omega\right)\Bigr]}{\E \Bigl[g^{\omega}_{ \delta}(0,y)\Bigr]} \nonumber \\
 &  = \epsilon \frac{1}{{\E \Bigl[g^{\omega}_{\delta}(0,y)\Bigr]}}\sum_{z \in A}\sum_{e \in V} {\rm Cov}\,\Bigl[\xi(z,e), f(\omega)\Bigr]J_{p^*_\epsilon}(z+e) \times \E \Bigl[g_\delta^{\omega}(0, z+y)\Bigr] \nonumber \\
& +  \sum_{z \in A}\sum_{e \in V} {\rm Cov}\,\Bigl[\xi(z,e), f(\omega)\Bigr]J_{p^*_\epsilon}(z+e) \times \frac{\E \Bigl[O_8(\epsilon^2)g_\delta^{\omega}(0, z+y)\Bigr]}{\E \Bigl[g_\delta^\omega (0,z+y)\Bigr]} \nonumber \\
& - \sum_{z \in A}\sum_{e \in V} {\rm Cov}\,\Bigl[\xi(z,e), f(\omega)\Bigr]J_{p^*_\epsilon}(z+e) \times \frac{\E \Bigl[O_9(\epsilon^2)g_\delta^{\omega}(0,y)\Bigr]}{\E \Bigl[g_\delta^\omega (0,y)\Bigl(1+O_7(\epsilon)\Bigr)\Bigr]} \nonumber \\
& - \sum_{z \in A}\sum_{e \in V} {\rm Cov}\,\Bigl[\xi(z,e), f(\omega)\Bigr]J_{p^*_\epsilon}(z+e) \times 
\frac{\E \Bigl[O_8(\epsilon^2)g_\delta^{\omega}(0, z+y)\Bigr]}{\E \Bigl[g_\delta^\omega (0,z+y)\Bigr]}
\times \frac{\E \Bigl[O_9(\epsilon)g_\delta^{\omega}(0,y)\Bigr]}{\E \Bigl[g_\delta^\omega (0,y)\Bigl(1+O_7(\epsilon)\Bigr)\Bigr]}
 \nonumber \\
& + O_3(\epsilon^{2-\eta}) \sum_{z \in A}\sum_{e \in V} {\rm Cov}\,\Bigl[\xi(z,e), f(\omega)\Bigr]  \nonumber \\
 & + \sum_{z \in A}\sum_{e \in V} {\rm Cov} \, \Bigl[\xi(z,e), f(\omega)\Bigr] J_{p^*_\epsilon}(z+e)\,\, \frac{\E[g_\delta^\omega (0,y)O_4(\epsilon^2)]}{\E[g^\omega_\delta (0,y)]} \nonumber\\
 & + \sum_{z \in A}\sum_{e \in V} {\rm Cov} \, \Bigl[\xi(z,e), f(\omega)\Bigr] \,\, \frac{\E[g_\delta^\omega (0,y)
O_5(\epsilon^{3-\eta})]}{\E[g^\omega_\delta (0,y)]} \nonumber\\ 
 & + \frac{\E[g_\delta^\omega (0,y)\bar{f}(\theta_y \omega) O_1(\epsilon^2)]}{\E[g_\delta^\omega (0,y)]},
\end{eqnarray}
where, by (\ref{fexp5.1}), we know that for each $\tau>0$
\begin{equation}
\label{fexp5.2}
|O_i(\tau)| \leq \frac{c_3}{(\delta \kappa)^{\rho(A)}} \tau, \quad \forall i=8,9.
\end{equation}   
Defining
\begin{eqnarray*}
& h_1:=\sum_{z \in A} \sum_{e \in V} \bar{\xi}(z,e) J_{p^*_\epsilon}(z+e) \times \frac{\E \Bigl[O_8 (\epsilon^2)g^\omega_\delta (0, z+y)\Bigr]}{E \Bigl[g^\omega_\delta (0, z+y)\Bigr]}.\nonumber \\
& h_2:=-\sum_{z \in A} \sum_{e \in V} \bar{\xi}(z,e) J_{p^*_\epsilon}(z+e) \times \frac{\E \Bigl[O_9 (\epsilon^2)g^\omega_\delta (0, y)\Bigr]}{E \Bigl[g^\omega_\delta (0, y)(1+O_7(\epsilon))\Bigr]}.\nonumber \\
& h_3:=-\sum_{z \in A} \sum_{e \in V} \bar{\xi}(z,e) J_{p^*_\epsilon}(z+e) \times \frac{\E \Bigl[O_8 (\epsilon^2)g^\omega_\delta (0, z+y)\Bigr]}{E \Bigl[g^\omega_\delta (0, z+y)\Bigr]} \times \frac{\E \Bigl[O_9 (\epsilon)g^\omega_\delta (0, y)\Bigr]}{E \Bigl[g^\omega_\delta (0, y)(1+O_7(\epsilon))\Bigr]}. \nonumber \\
& h_4:=\sum_{z \in A} \sum_{e \in V} \bar{\xi}(z, e) O_3(\epsilon^{2-\eta}). \nonumber \\
& h_5:=\sum_{z \in A} \sum_{e \in V} \bar{\xi}(z,e) J_{p^*_\epsilon}(z+e) \times \frac{\E \Bigl[O_4 (\epsilon^2)g^\omega_\delta (0,y)\Bigr]}{E \Bigl[g^\omega_\delta (0,y)\Bigr]}.\nonumber \\
& h_6:=\sum_{z \in A} \sum_{e \in V} \bar{\xi}(z,e) \times \frac{\E \Bigl[O_5 (\epsilon^{3-\eta})g^\omega_\delta (0,y)\Bigr]}{E \Bigl[g^\omega_\delta (0,y)\Bigr]}\nonumber \\
& h_7:=\frac{g^\omega_\delta (-y,0)\tilde{O}_1(\epsilon^2)}{\E[g^\omega_\delta (0,y)]}, \quad {\rm with} \quad \tilde{O}_1(\epsilon^2)(\omega):=O_1(\epsilon^2)\Bigl(\theta_{(-y)}\omega\Bigr).\nonumber \\
& h_8:=-\frac{\E[g^\omega_\delta (0,y)O_1(\epsilon^2)]}{\E[g^\omega_\delta (0,y)]},\nonumber \\
\end{eqnarray*}
we can rewrite (\ref{fexp6}) in the following way
\begin{eqnarray}
\label{fexp7}
&\frac{\E \Bigl[\left(g^{\omega}_{\delta}(0,y)-g^{\omega^{A+y}}_{\delta}(0,y)\right)\bar{f} \left(\theta_y \omega\right)\Bigr]}{\E \Bigl[g^{\omega}_{ \delta}(0,y)\Bigr]} \nonumber \\
 & = \epsilon \frac{1}{{\E \Bigl[g^{\omega}_{\delta}(0,y)\Bigr]}}\sum_{z \in A}\sum_{e \in V} {\rm Cov}\,\Bigl[\xi(z,e), f(\omega)\Bigr]J_{p^*_\epsilon}(z+e)  \E \Bigl[g_\delta^{\omega}(0, z+y)\Bigr]+E[h_\delta f], \end{eqnarray}
where $h_\delta:=\sum_{i=1}^8 h_i$. Now, in order to show that (\ref{expansion1a}) is satisfied, it is necessary to justify that for any $z \in A$ and $e\in V$, $J_{p^*_\epsilon} (z+e)$ is bounded. If we use (\ref{fexp2.1}), (\ref{fexp3.1}), (\ref{fexp4.1}), (\ref{fexp4.2}), (\ref{fexp5.1}), (\ref{fexp5.2})
and the fact that for each $e \in V$ and $z \in A$, $\bar{\xi}(z,e)$ is bounded by $2$, it is easy to deduce that there exists $\epsilon_0>0$ and an constant $c_9=c_9(\eta)$ such that  for all $0<\epsilon \leq \epsilon_0$,
whenever $\delta$ is close enough to $1$, the following inequality is satisfied for all $1\le i\le 8$,
\begin{equation}
\label{fexp8}
\Bigl|\E[|h_i| f] \Bigr| \leq \Bigl| O(\epsilon^{2-\eta}) \Bigr| \E[|f|] , \quad {\rm with} \,\,\Bigl|O(\epsilon^{2-\eta})\Bigr| \leq c_9 \epsilon^{2-\eta}.
\end{equation}
In the case of $h_7$, if we apply independence and use  (\ref{constant1}), (\ref{Green5}) and (\ref{fexp2.1}) , one can deduce that for $0<\epsilon \leq \epsilon_0$ exists a constant $c_{10}>0$ such that
\begin{eqnarray}
\label{fexp9}
& \Bigl|\E[|h_7| f] \Bigr| \leq \frac{1}{\E[g^\omega_\delta (0,y)]}\times \E \left[g^{\omega^{A+y}}_\delta (0,y) \cdot |f(\theta_y \omega)| \cdot \frac{g^{\omega}_\delta (0,y)}{g^{\omega^{A+y}}_\delta(0,y)} |O_1(\epsilon^2)|\right] \nonumber \\
& \leq \Bigl| O(\epsilon^2) \Bigr | \E[|f|],
\end{eqnarray}
where $|O(\epsilon^2)|\leq c_{10} \epsilon^2$.
Finally, with the help of (\ref{fexp8}) and (\ref{fexp9}),  Proposition \ref{expansion} is easily proven.

\section{Proof of Theorem \ref{theorem1}}
\label{proof}
In this section we will prove Theorem \ref{theorem1}. Note that, with the help of (\ref{nud}) and Proposition \ref{expansion}, there exists an $\epsilon_0>0$ and a constant $c_4$ such that for $\delta$ close enough
to $1$, for all  $y \in \Z^d$ and $0<\epsilon \leq \epsilon_0$ we have 
that

\begin{eqnarray*}
& \int f d\mu_{\delta}=\frac{\sum_{y \in \Z^d} \left(\E [g^\omega_\delta (0,y)] \E[f]+\epsilon 
\sum\limits_{z \in A,e \in V} {\rm Cov}\,\Bigl[\xi(z,e), f(\omega)\Bigr] \nonumber  
J_{p^*_\epsilon}(z+e)
 \E \Bigl[g_{\delta}^{\omega}(0,z+y)\Bigr]
\nonumber +\E [g^\omega_\delta (0,y)]\E[h_\delta f]
\right)}{\sum_{y \in \Z^d}\E[g^\omega_\delta (0,y)]}\\
 & = \E[f]+\epsilon\sum\limits_{z \in A,e \in V} {\rm Cov}\,\Bigl[\xi(z,e), f(\omega)\Bigr]
J_{p^*_\epsilon}(z+e)
+\E[h_\delta f],
\end{eqnarray*} 
where for all $0<\epsilon \leq \epsilon_0$, 
we have that  

\begin{equation}
\label{bbound}
|\E[|h_\delta| f]| \leq \Bigl |O(\epsilon^{2-\eta}) \Bigr|\E[|f|],
\end{equation}
with $|O(\epsilon^{2-\eta})|\le c_4\epsilon^{2-\eta}$.
Taking now the limit when $\delta\to 1$,
by Proposition \ref{polinomial} we conclude that

$$
\int fd\mathbb Q=\E[f]+\epsilon\sum\limits_{z \in A,e \in V} {\rm Cov}\,\Bigl[\xi(z,e), f(\omega)\Bigr]
J_{p^*_\epsilon}(z+e)
+\int fd\mathbb V,
$$
where by (\ref{bbound}), $\mathbb V$ is a signed measure satisfying

$$
\left|\int fd|\mathbb V|\right|\le \Bigl|O(\epsilon^{2-\eta}) \Bigr|\mathbb E[|f|],
$$
where $|\mathbb V|$ is the variation of $\mathbb V$.
Hence, the restriction of $\mathbb Q$ to $A$ is absolutely
continuous with respect to $\mathbb P_A$, from where we can conclude
that  there is a function $h$ such that $\mathbb E[|h|]<\infty$,
and such that for every bounded and continuous function $f$ one has that

\begin{equation}
\label{fh}
|\mathbb E[f|h|]|\le \Bigl|O(\epsilon^{2-\eta}) \Bigr|\mathbb E[|f|]
\end{equation}
and

$$
\int fd\mathbb Q=\E[f]+\epsilon\sum\limits_{z \in A,e \in V} {\rm Cov}\,\Bigl[\xi(z,e), f(\omega)\Bigr]
J_{p^*_\epsilon}(z+e)
+\mathbb E[fh].
$$
Now, approximating any function $f$ in $L_1$ by bounded continuous functions, it is
easy to check that in fact (\ref{fh}) is satisfied for every $f\in L_1$. Therefore,
$h$ is bounded, from where we conclude the proof of Theorem \ref{theorem1}.

\section{Proof of Corollary \ref{corollary1}}
\label{end}
Here we prove Corollary \ref{corollary1}. It is enough
to show that  $J_{p^*_\epsilon}$ [c.f. (\ref{jepsilon})] is
well approximated by $J_{p^*_0}$. By
standard Fourier inversion formulas (see for
example Spitzer \cite{Sp64}) we can conclude  that for each $z\in\mathbb Z^d$ and $e\in V$ one has that

\begin{eqnarray}
\label{p01}
& J_{p^*_\epsilon}(z+e)=\frac{1}{(2 \pi)^d}  
\left( \prod_{j=1}^d \left(\frac{p^\epsilon(-e_j)}{p^\epsilon(e_j)}\right)^{\frac{z_j+e_j}{2}} -1\right) \int_{[0,2\pi]^d}\frac{\cos \left(\sum_{j=1}^d (z_j+e_j) x_j\right)}{1-2 \sum_{j=1}^d \sqrt{p^\epsilon (e_j)p^\epsilon (-e_j)}\, \cos(x_j)}\, \prod dx_j \nonumber\\ & + \frac{1}{(2 \pi)^d} \int_{[0,2\pi]^d}\frac{\cos \left(\sum_{j=1}^d (z_j+e_j) x_j\right)-1}{1-2 \sum_{j=1}^d \sqrt{p^\epsilon (e_j)p^\epsilon (-e_j)}\, \cos(x_j)}\, \prod dx_j,
\end{eqnarray}
where for each $1\le j\le d$, $z_j$ is the $j$-th coordinates of $z$.
When $p_0=0$, we can conclude from (\ref{p01}) that

\begin{equation}
\nonumber
J_{p^*_\epsilon}(z+e)=
\begin{cases}
J_{p_0^*}(z+e)+O(\epsilon\log\epsilon)&\quad{\rm if}\quad d=2\\
J_{p_0^*}(z+e)+O(\epsilon)&\quad{\rm if}\quad d\ge 3,
\end{cases}
\end{equation}
For the case $p_0\ne 0$, a simpler estimation gives us that
for any dimension $d\ge 2$ one has that
$$
J_{p^*_\epsilon}(z+e)=
J_{p_0^*}(z+e)+O(\epsilon).
$$

\end{document}